\documentclass[reqno]{amsart}
\usepackage{amssymb}
\usepackage[usenames, dvipsnames]{color}
\usepackage{txfonts}
\usepackage{hyperref}
\usepackage{verbatim}
\usepackage[pagewise]{lineno}
\usepackage{marginnote}
\usepackage{todonotes}
\usepackage[normalem]{ulem}
\usepackage{cancel}
\usepackage{soul}

\theoremstyle{plain}
\newtheorem{theorem}{Theorem}[section]
\newtheorem{lemma}[theorem]{Lemma}
\newtheorem{corollary}[theorem]{Corollary}

\theoremstyle{definition}

\newtheorem{condition}[theorem]{Condition}

\theoremstyle{remark}
\newtheorem{remark}[theorem]{Remark}

\newcommand{\bR}{{\mathbb R}}

\newcommand{\pf}{\noindent {\bf Proof. \hspace{2mm}}}
\def\ve{\varepsilon}

\def\th{\theta}

\def\G{\Gamma}

\def\Dl{\Delta}
\def\lt{\left}
\def\rt{\right}
\def\les{\lesssim}

\def\i{\infty}

\def\p{\partial}
\def\f{\frac}
\def\na{\nabla}
\def\al{\alpha}

\def\O{\Omega}

\def\q{\quad}
\def\qq{\qquad}

\def\blue{\color{blue}}
\def\be{\begin{equation}}
\def\ee{\end{equation}}
\def\bes{\begin{equation*}}
\def\ees{\end{equation*}}
\def\bali{\begin{aligned}}
\def\eali{\end{aligned}}

\numberwithin{equation}{section}

\makeatletter
\def\dashint{\operatorname%
{\,\,\text{\bf--}\kern-.98em\DOTSI\intop\ilimits@\!\!}}
\makeatother

\begin{document}

%\linenumbers

\title[One component regularity criteria of MHD-Boussinesq]{One component regularity criteria for the axially symmetric MHD-Boussinesq system: criteria on the swirl component of the velocity}

\author[Z. Li]{Zijin Li}
\address[Z. Li]{School of Mathematics and Statistics, Nanjing University of Information Science $\&$ Technology}

\email{zijinli@nuist.edu.cn}

\thanks{Z. Li is supported by Natural Science Foundation of Jiangsu Province (No. SBK2020040176), National Natural Science Foundation of China (No. 12001285), Double Innovation Scheme of Jiangsu Province and the Startup Foundation for Introducing Talent of NUIST (No. 2019r033).}
%\thanks{H. Dong was partially supported by the NSF under agreement DMS-1600593.}

\author[X. Pan]{Xinghong Pan}
\address[X. Pan]{Department of Mathematics, Nanjing University of Aeronautics and Astronautics, Nanjing 211106, China}

\email{xinghong\_87@nuaa.edu.cn}

\thanks{X. Pan is supported by Natural Science Foundation of Jiangsu Province (No. SBK2018041027), Double Innovation Scheme of Jiangsu Province and National Natural Science Foundation of China (No. 11801268).}



\begin{abstract}
In this paper, we consider regularity criteria of a class of 3D axially symmetric MHD-Boussinesq systems without magnetic resistivity or thermal diffusivity. Under some Prodi-Serrin type critical assumptions on the horizontal angular component of the velocity, we will prove that strong solutions of the axially symmetric MHD-Boussinesq system can be smoothly extended beyond the possible blow-up time $T_*$ if the magnetic field contains only the horizontal swirl component. No a priori assumption on the magnetic field and the temperature fluctuation is imposed.

%Based on the article I of this series \cite{LiPan2020-1}, in this paper we give a critical Prodi-Serrin type regularity criterion with $u^\th$ component of a class of 3D axially symmetric MHD-Boussinesq system. That is, if
%\[
%\int_0^{T_*}\left\|\f{u^\th}{r^{s}}(t,\cdot)\right\|_{L^p}^{q}dt<\infty,\quad \text{where}\quad \f{3}{p}+\f{2}{q}\leq 1+s\text{ for }\f{3}{1+s}<p\leq\infty,\text{ with } s\geq0,
%\]
%or there exists an $\ve_0<<1$, such that
%\[
%\sup_{0\leq t\leq T_*}\left\|\f{u^\th}{r^{s}}(t,\cdot)\right\|_{L^{\f{3}{1+s}}}<\ve_0,
%\]
%then the strong solution of viscous MHD-Boussinesq system can be smoothly extended beyond time $T_*$.

\end{abstract}
%\today
\maketitle

\section{Introduction}
In this paper, we study the 3D MHD-Boussinesq system without magnetic resistivity and thermal diffusivity:
\begin{equation}\label{MB}
\left\{
\begin{aligned}
&
\p_tu+u\cdot\nabla u+\nabla p-\mu\Delta u=h\cdot\nabla h+\rho e_3,\\
&\p_t h+u\cdot\nabla h-h\cdot\nabla u=0,\\
&\p_t\rho+u\cdot\na\rho=0,\\
&\nabla\cdot u=\nabla\cdot h=0.\\
\end{aligned}
\right.
\end{equation}
Here $u\in\bR^3$ is the velocity and $h\in\bR^3$ is the magnetic field, while $p\in\bR$ and $\rho\in\bR$ represent the pressure and the temperature fluctuation, respectively. $e_3=(0,0,1)^T$ is the unit vector in the vertical direction and $\mu>0$ stands for the viscosity constant, which is assumed to be $1$ without loss of generality in the following.

Physically, equation \eqref{MB}$_1$ describes the conservation law of the momentum with the influence of buoyant effect $\rho e_3$, while \eqref{MB}$_2$ is the non-resistive Maxwell-Faraday equation which describes the Faraday's law of induction. The third line of \eqref{MB} represents the ideal temperature fluctuation, while the fourth line describes the incompressibility of the fluid and Gauss's law for magnetism.
The MHD-Boussinesq system, which models the convection of an incompressible conductive flow driven by the Lorenz force and buoyant effect of a thermal field, plays an important role in atmospheric science and geophysical applications. It is closely related to a type of the Rayleigh-B\'{e}nard convection, which occurs in a horizontal layer of conductive fluid heated from below, with a presence of a magnetic field. For detailed physical background, we refer readers to \cite{Pedlosky1987, Marchioro1994, Majda2003, Mulone2003}.

Our main result and its proof will be presented in the cylindrical coordinates $(r,\,\th,\,z)$. For $x=(x_1,\,x_2,\,x_3)\in\mathbb{R}^3$, we denote
\[
r=\sqrt{x_1^2+x_2^2},\q \th=\arctan\frac{x_2}{x_1},\q z=x_3,
\]
and
\[
e_r=\left(\frac{x_1}{r},\frac{x_2}{r},0\right),\quad e_\th=\left(-\frac{x_2}{r},\frac{x_1}{r},0\right),\quad e_z=(0,0,1).
\]
We say a solution of \eqref{MB} is axially symmetric, if and only if
\[
\left\{
\begin{split}
u&=u^r(t,r,z)e_r+u^{\th}(t,r,z)e_{\th}+u^z(t,r,z)e_z,\\
h&=h^r(t,r,z)e_r+h^{\th}(t,r,z)e_{\th}+h^z(t,r,z)e_z,\\
\rho&=\rho(t,r,z),\\
\end{split}
\right.
\]
satisfy the system \eqref{MB}. By the local existence and uniqueness results, it is clear that one only needs to assume $h^r_0=h^z_0\equiv0$, then vanishing of $h^r$ and $h^z$ holds for all time (see \cite{Larios2017}). In this case, \eqref{MB} can be rewritten as
\be\label{MHD-BOUS}
\left\{
\begin{aligned}
&\p_tu^r+(u^r\p_r+u^z\p_z)u^r -\frac{(u^\th)^2}{r}+\p_r P=-\frac{(h^\th)^2}{r}+\left(\Delta-\frac{1}{r^2}\right)u^r, \\[4mm]
&\p_tu^\th+(u^r\p_r+u^z\p_z) u^\th+\frac{u^\th u^r}{r}=\left(\Delta-\frac{1}{r^2}\right)u^\th , \\[4mm]
&\p_tu^z+(u^r\p_r+u^z\p_z)u^z+\p_z P=\Delta u^z+\rho ,                                    \\[4mm]
&\p_th^\th+(u^r\p_r+u^z\p_z)h^\th-\frac{h^\th u^r}{r}=0,\\[4mm]
&\p_t\rho+(u^r\p_r+u^z\p_z)\rho=0,\\[4mm]
&\nabla\cdot u=\p_ru^r+\frac{u^r}{r}+\p_zu^z=0,
\end{aligned}
\right.
\ee
where $P:=p+\f{1}{2}|h^\th|^2$. To state the regularity theorem of the initial value problem of the axially symmetric solution of \eqref{MB}, we present here a Prodi-Serrin type condition on the horizontal swirl component of the velocity:
\begin{condition}\label{COND1}
For any $s \geq 0$,
\be\label{C1.305}
\int_0^{T_*}\left\|\f{u^\th}{r^{s}}(t,\cdot)\right\|_{L^p}^{q}dt<\infty,\quad \text{where}\quad \f{3}{p}+\f{2}{q}\leq 1+s\quad\text{for}\quad\f{3}{1+s}<p\leq\infty.
\ee
\end{condition}

\qed

Meanwhile, for the borderline case $p=\f{3}{1+s}$, we assume that:

\begin{condition}\label{COND2}
For any $s \geq 0$,
\be\label{C1.306}
\sup_{0\leq t\leq T_*}\left\|\f{u^\th}{r^{s}}(t,\cdot)\right\|_{L^{\f{3}{1+s}}}<\ve_0,
\ee
where $\ve_0=\ve_0\left(s,\,ru^\th_0\right)<<1$ will be decided in the proof of Theorem \ref{th0}.
\end{condition}

\qed

Now we are ready for the main result:
\begin{theorem}\label{th0}
 Let $(u,h,\rho)$ be an axially symmetric solution of \eqref{MB} whose initial data $(u_0,h_0,\rho_0)\in H^m(\mathbb{R}^3)$, for $m\geq3$, and $\na\cdot u_0=h^r_0=h^z_0=0$. Then $(u,h,\rho)$ can be smoothly extended beyond $T_*$ if and only if Condition \ref{COND1} or Condition \ref{COND2} holds.
\end{theorem}

\qed

Critical regularity criteria of incompressible fluid dynamic systems date back to pioneer works of  G. Prodi \cite{Prodi1959} and J. Serrin \cite{Serrin1962, Serrin1963} around the 1960s, where the famous Prodi-Serrin criterion for 3D Naiver-Stokes equations is given. Readers can see \cite{Fabes1972, Giga1986, Struwe1988, Takahashi1990, Veiga1995} for more regularity results on the Navier-Stokes equations.

If the fluid (say e.g. plasma) is affected by the Lorentz force, then the Navier-Stokes system is generalized to the magnetohydrodynamics system. Many fruitful studies and researches on the partial regularity and blow up criteria of the MHD system had been achieved in recent years, see e.g. \cite{HeXin2005-1,ChenMiaoZhang2008, RWXZ:2014JFA, LXZ:JDE2015, Lei2015-2} and references therein. Moreover, if the fluid is influenced by the temperature, then the fluid equations can be modeled by the classical Boussinesq system. See \cite{HL:2005DCDS, Chae:2006ADV, HR:2010AIHP, AHK:2011DCDS, Larios2013, CaoWu2013} and references therein for more regularity results on the Boussinesq system.

Recently there are more and more studies concerning the full 3D MHD-Boussinesq system. We refer readers to \cite{Larios2017, LiuBianPu2019, BianPu2020, Pan2019}, etc. for the regularity criteria, local and global well-posedness of weak and strong solutions of the MHD-Boussinesq system. The local well-posedness results were proved in Larios-Pei \cite{Larios2017}.  If a nonlinear damping term was added in the momentum equations, Liu-Bian-Pu \cite{LiuBianPu2019} proved the global well-posedness of strong solutions. Recently, Bian-Pu \cite{BianPu2020} proved the global regularity of axially symmetric large solutions to the MHDB system \eqref{MHD-BOUS} without the horizontal swirl component $u^\th$ of the velocity under the assumption that the support of the initial thermal fluctuation is away from the $z$-axis and its projection on to the $z$-axis is compact. Later, this result was improved in Pan \cite{Pan2019} by removing the ``support set'' assumption on the initial data of the thermal fluctuation.
%In \cite{LiPan2020}, the authors derive several Prodi-Serrin type criteria and BKM type criteria for viscous and inviscid MHD-Boussinesq system.
%Also a global well-posed result for 3D MHD-Boussinesq system with partially small initial data is derived. In \cite{LiPan2020-1}, the authors gave some critical regularity criteria, which depend only one component of the vorticity, for a class of 3D MHD-Boussinesq system that is viscous or inviscid.

Throughout the paper, $C_{a,b,c,...}$ denotes a positive constant depending on $a,\,b,\, c,\,...$ which may be different from line to line. We also apply $A\lesssim B$ to denote $A\leq CB$. Meanwhile, $A\simeq B$ means both $A\lesssim B$ and $B\lesssim A$. $[\mathcal{A},\,\mathcal{B}]=\mathcal{A}\mathcal{B}-\mathcal{B}\mathcal{A}$ denotes the communicator of the operator $\mathcal{A}$ and the operator $\mathcal{B}$. For any Banach space $X$, we say $v:\,[0,T]\times\mathbb{R}^3\to\mathbb{R}$ belongs to the Bochner space $L^p(0,T;X)$, if
\[
\|v(t,\cdot)\|_{X}\in L^p(0,T),
\]
and we usually use $L^p_TX$ for short notation of $L^p(0,T;X)$.

 Our proof of the main result in this paper consists of the following steps: First, we investigate a reformulated system \eqref{REFORM} which is motivated by \cite{HR:2010AIHP,ChenFangZhang2017} and derive a closed $L^\infty_TL^2\cap L^2_TH^1$ estimate of \eqref{REFORM} under the condition \eqref{C1.305} or \eqref{C1.306}. Based on this estimate, we further derive the $L^\infty_TL^2\cap L^2_TH^1$ estimate of $\nabla u$, \eqref{EVORTI}. Using the maximal regularity result of the heat flow, then we arrive the $L^1_TL^\infty$ estimate of $\nabla u$ \eqref{ENU}, then the $L^1_TL^\infty$ estimates of $\nabla\times h$ and $\nabla\rho$, \eqref{ENH} and \eqref{ERHO}, follow. Finally, using these $L^1_TL^\infty$ estimates, the estimates of higher-order norms of the solution follow from a classical communicator estimate by Kato-Ponce \cite{Kato1988}.

The remaining of this paper is organized as follows. In Section 2, we provide some useful Lemmas concerning interpolation inequalities, some $L^p$ boundedness of a singular operator which is related to the temperature fluctuation in the MHD-Boussinesq system, a Hardy type inequality, and logarithmic imbedding inequalities. Finally, in Section 3, we provide the proof of regularity criterion Theorem \ref{th0}.

\section{Preliminaries}
At the beginning, let us introducte the well-known $Gagliardo-Nirenberg$ interpolation inequality. We list here without proof.
\begin{lemma}[Gagliardo-Nirenberg]\label{LEMGN}
Fix $q,r\in[1,\i]$ and $j,m\in\mathbb{N}\cup\{0\}$ with $j\leq m$. Suppose that $f\in L^q(\mathbb{R}^d)\cap\dot{W}^{m,r}(\mathbb{R}^d)$ and there exists a real number $\al\in[j/m,1]$ such that
\[
\frac{1}{p}=\frac{j}{d}+\al\left(\frac{1}{r}-\frac{m}{d}\right)+\frac{1-\al}{q}.
\]
Then $f\in\dot{W}^{j,p}(\mathbb{R}^d)$ and there exists a constant $C>0$ such that
\[
\|\na^jf\|_{L^p(\mathbb{R}^d)}\leq C\|\na^m f\|^\al_{L^r(\mathbb{R}^d)}\|f\|^{1-\al}_{L^q(\mathbb{R}^d)},
\]
except the following two cases:

(i) $j=0$, $mr<d$ and $q=\infty$; (In this case it is necessary to assume also that either $|u|\to 0$ at infinity, or $u\in L^s(\mathbb{R}^d)$ for some $s<\infty$.)

(ii) $1<r<\infty$ and $m-j-d/r\in\mathbb{N}$. (In this case it is necessary to assume also that $\alpha<1$.)

\end{lemma}
In the following we state a useful space-time interpolation which is frequently used in the research of Navier-Stokes equations:
\begin{lemma}\label{LEM2.3}
If $u \in L^{\infty}\left(0, T ; L^2(\mathbb{R}^3)\right) \cap L^{2}\left(0, T ; \dot{H}^1(\mathbb{R}^3)\right)$, then
\be\label{Vorti222}
u \in L^{q}\left(0, T ; L^{p}(\mathbb{R}^3)\right),
\ee
where
\[
\frac{2}{q}+\frac{3}{p}\geq\frac{3}{2}, \quad 2 \leq p \leq 6.
\]
\end{lemma}
\begin{proof}
The Sobolev inequality implies $u\in L^{2}\left(0, T ; L^6(\mathbb{R}^3)\right)$. Then we interpolate the $L^{s}$ norm between $L^{2}$ and $L^{6}$ to derive
\[
\|u\|_{L^{p}} \leq\|u\|_{L^{2}}^{(6-p) / 2 p}\|u\|_{L^{6}}^{(3 p-6) / 2 p}.
\]
This indicates
\[
\int_{0}^{T}\|u\|_{L^{p}}^{q}dt \leq \int_{0}^{T}\|u\|_{L^{2}}^{(6-p) q / 2 p}\|u\|_{L^{6}}^{(3 p-6) q / 2 p}dt.
\]
Since $u\in L^{\infty}\left(0,T;L^{2}\right)\cap L^{2}\left(0,T; L^{6}\right)$, the integral on the right-hand side of the above inequality is bounded when $(3 p-6) q / 2 p \leq 2$, which corresponds to
\[
\frac{2}{q}+\frac{3}{p} \geq \frac{3}{2}.
\]
\end{proof}

Next, we focus on the following estimates of a triple product form with commutator:
\begin{lemma}\label{LEMET1}
Let $m\in\mathbb{N}$, $m\geq 2$, and $f,g,k\in C^\infty_0(\mathbb{R}^3)$. Then the following estimate holds:
\be\label{E1}
\begin{split}
\left|\int_{\mathbb{R}^3}[\nabla^m,\,f\cdot\nabla]g\nabla^m kdx\right|\leq&\,C\,\left\|\nabla^{m}(f,g,k)\right\|_{L^2}^2\|\nabla (f,\,g)\|_{L^\infty}.
\end{split}
\ee
\end{lemma}
\begin{proof}
Applying H\"{o}lder's inequality, one derives
\be\label{2.77777}
\left|\int_{\mathbb{R}^3}[\nabla^m,\,f\cdot\nabla]g\nabla^m kdx\right|\leq\|[\nabla^m,\,f\cdot\nabla]g\|_{L^{2}}\|\nabla^mk\|_{L^2}.
\ee
Due to the commutator estimate by Kato-Ponce \cite{Kato1988}, it follows that
\be\label{2.88888}
\|[\nabla^m,\,f\cdot\nabla]g\|_{L^{2}}\leq C\left(\|\nabla f\|_{L^\infty}\|\nabla^mg\|_{L^2}+\|\nabla g\|_{L^\infty}\|\nabla^m f\|_{L^2}\right).
\ee
Then \eqref{E1} follows from plugging \eqref{2.88888} into \eqref{2.77777}.

\end{proof}

\begin{lemma}\label{LEMET11}
Denote $\mathcal{L}=\left(\Delta+\frac{2}{r} \partial_{r}\right)^{-1} \frac{\partial_{r}}{r}$ and $\tilde{\mathcal{L}}=\left(\Delta+\frac{2}{r} \partial_{r}\right)^{-1} \frac{\partial_{z}}{r} .$ Suppose $\rho \in H^{2}\left(\mathbb{R}^{3}\right)$ be axisymmetric, then for every $p \in[2,+\infty),$ there
exists an absolute constant $C_{p}>0$ such that
\[
\|\mathcal{L} \rho\|_{L^{p}} \leq C_{p}\|\rho\|_{L^{p}}, \quad\|\tilde{\mathcal{L}} \rho\|_{L^{p}} \leq C_{p}\|\rho\|_{L^{p}}.
\]
Moreover, for any smooth axisymmetric function $f,$ we have the identity
\[
\mathcal{L} \partial_{r} f=\frac{f}{r}-\mathcal{L}\left(\frac{f}{r}\right)-\partial_{z} \tilde{\mathcal{L}} f.
\]
\end{lemma}

\begin{proof} The detailed proof can be found in Proposition 3.1, 3.2 and Lemma 3.3 in \cite{HR:2010AIHP}. We omit the details here.

\end{proof}

\begin{lemma}\label{LEMET12}
Define $\O:=\f{w^\th}{r}$. For $1<p<+\infty$, there exists an absolute constant $C_{p}>0$ such that
$$
\left\|\nabla \frac{u^{r}}{r}\right\|_{L^{p}} \leq C_{p}\|\Omega\|_{L^{p}}.
$$
\end{lemma}

The proof of this lemma can be founded in many literatures, such as \cite{Lei2015-2} ( equation (A.5)) and  \cite{Miao-Zheng2013} (Proposition 2.5).

\qed

Next we give a Sobolev-Hardy inequality. We omit the detailed proof since it could be found in the Lemma 2.4 of \cite{ChenFangZhang2017}.
\begin{lemma}\label{CFZ}
Set $\mathbb{R}^{n}=\mathbb{R}^{k} \times \mathbb{R}^{n-k}$ with $2 \leq k \leq n,$ and write $x=\left(x^{\prime}, z\right) \in \mathbb{R}^{k} \times \mathbb{R}^{n-k} .$ For
$1<q<n, 0 \leq \theta \leq q$ and $\theta<k,$ let $q_{*} \in\left[q, \frac{q(n-\theta)}{n-q}\right] .$ Then there exists a positive constant $C=C(\theta, q, n, k)$ such that for all $f \in C_{0}^{\infty}\left(\mathbb{R}^{n}\right),$
\[
\int_{\mathbb{R}^{n}} \frac{|f|^{q_{*}}}{\left|x^{\prime}\right|^{\theta}} d x \leq C\|f\|_{L^q}^{\frac{n-\theta}{q^{*}}-\frac{n}{q}+1}\|\nabla f\|_{L^q}^{\frac{n}{q}-\frac{n-\theta}{q *}}.
\]
In particular, we pick $n=3, k=2, q=2, q_{*} \in[2,2(3-\theta)],$ and assume $0 \leq \theta<2, r=$ $\sqrt{x_{1}^{2}+x_{2}^{2}} .$ Then there exists a positive constant $C=C\left(q_{*}, \theta\right)$ such that for all $f \in C_{0}^{\infty}\left(\mathbb{R}^{n}\right)$
\be\label{CFZEST}
\left\|\frac{f}{r^{\frac{\theta}{q *}}}\right\|_{L^{q_{*}}} \leq C\|f\|_{L^2}^{\frac{3-\theta}{q *}-\frac{1}{2}}\|\nabla f\|_{L^2}^{\frac{3}{2}-\frac{3-\theta}{q *}}.
\ee
\end{lemma}

\qed

%\begin{corollary}\label{CORMRP}
%Assume that the Lebesgue exponents $p_{1}$ and $p_{2}$ fulfill $0 \leq \f{1}{p_{2}}-\f{1}{p_{1}}<\f{1}{d}$, and that $1<r_{2}<r_{1}<\infty$ are interrelated by the relation
%\[
%\frac{2}{r_{2}}+\frac{d}{p_{2}}=1+\frac{2}{r_{1}}+\frac{d}{p_{1}}.
%\]
%Let us define the operator $\mathcal{A}_{1}$ by
%\[
%\mathcal{A}_{1}: \quad f \mapsto \int_{0}^{t} \nabla e^{(t-s) \Delta} f(s, \cdot) d s
%\]
%Then $\mathcal{A}_{1}$ is bounded from $L^{r_{2}}\left(0,T; L^{p_{2}}(\mathbb{R}^{d})\right)$ to $L^{r_{1}}\left(0, T; %L^{p_{1}}(\mathbb{R}^{d})\right)$ for every $T \in(0,\infty]$, and
%there holds
%\be
%\left\|\mathcal{A}_{1} f\right\|_{L_{T}^{r_{1}}\left(L^{p_{1}}\right)} \leq C\|f\|_{L_{T}^{r_{2}}\left(L^{p_{2}}\right)}
%\ee
%for a suitable constant $C>0$ depending only on the space dimension $d \geq 1$ and on $p_{1}, r_{1}, p_{2}, r_{2}$.
%\end{corollary}

%\qed
 Using the Biot-Savart law and the $L^p$ boundedness of  Calderon-Zygmund singular integral operators, we have the following lemma whose detailed proof can be found for example in \cite{CL:2002MZ, Chen-Zhang2007}.
\begin{lemma}
Let $u=u^re_r+u^\th e_\th+u^z e_z$ be an axially symmetric vector field, $w=\nabla\times u=w^re_r+w^\th e_\th+w^z e_z$ and $b=u^r e_r+u^z e_z$. Then we have
\be\label{Biot1}
\|\nabla u\|_{L^p}\leq C_p\|w\|_{L^p},\q \|\nabla^2 u\|_{L^p}\leq C_p\|\na w\|_{L^p}
\ee
and
\be\label{Biot2}
\|\nabla b\|_{L^p}\leq C_p\|w^\th\|_{L^p}, \q \|\nabla^2 b\|_{L^p}\leq C_p\left(\|\na w^\th\|_{L^p}+\left\|\f{w^\th}{r}\right\|_{L^p}\right)
\ee
for all $1<p<\infty$.
\end{lemma}

\begin{lemma}[Maximal $L^r_TL^p$ regularity for the heat flow]\label{MRP}
Let us define the operator $\mathcal{A}_{2}$ by the formula
\[
\mathcal{A}_{2}: \quad f \longmapsto \int_{0}^{t} \nabla^{2} e^{(t-s) \Delta} f(s, \cdot) d s.
\]
Then $\mathcal{A}_{2} \text { is bounded from } L^{r_{2}}\left(0, T; L^{p_{2}}(\mathbb{R}^{d})\right) \text { to } L^{r_{2}}\left( 0, T; L^{p_{2}}(\mathbb{R}^{d})\right)$ for every  $T\in(0,\infty]$ and
$1<p_{2}, r_{2}<\infty$. Moreover, there holds:
\be\label{2.10PR}
\left\|\mathcal{A}_{2} f\right\|_{L_{T}^{r_{2}}\left(L^{p_{2}}\right)} \leq C\|f\|_{L_{T}^{r_{2}}\left(L^{p_{2}}\right)}.
\ee
\end{lemma}

Finally, we recall the following logarithmic imbedding inequality which is proved in \cite{Kozono2000}.

\begin{lemma}\label{LEM2.5}
Let $1<p<\infty$ and $s>d/p$. There exists a constant $C=C_{d,p,s}$ such that the estimate
\be\label{2.28b}
\|f\|_{L^\infty(\mathbb{R}^d)}\leq C\left(1+\|f\|_{BMO}\log(e+\|f\|_{W^{s,p}(\mathbb{R}^d)})\right)
\ee
holds for all $f\in W^{s,p}(\mathbb{R}^d)$.
\end{lemma}
In this paper, the following corollary of Lemma \ref{LEM2.5} is more convenient for us. That is:
\begin{corollary}\label{COR2.6}
For any divergence free vector field $g:\mathbb{R}^3\,\to\,\mathbb{R}^3$ such that $g\in H^3(\mathbb{R}^3)$, the following estimate holds:
\be\label{2.29b}
\|\na g\|_{L^\i(\mathbb{R}^3)}\les\,1+\|\na\times g\|_{BMO }\log\left(e+\|g\|_{H^3(\mathbb{R}^3)}\right).
\ee
\end{corollary}
\begin{proof} Using Fourier transform and noting that
\[
\xi\otimes\hat{g}=-\f{\xi}{|\xi|}\otimes\left(\f{\xi}{|\xi|}\times(\xi\times\hat{g})\right)
\]
provided $\xi\cdot\hat{g}\equiv0$, \eqref{2.29b} is proven by combining the estimate \eqref{2.28b} and the fact that the Riesz operator is bounded in the BMO space.

\end{proof}
\section{Proof of Theorem \ref{th0}}
In this section, we focus on the proof of Theorem \ref{th0}. Denoting $\Gamma:=ru^\th$ and $H:=\f{h^\th}{r}$, by \eqref{MHD-BOUS}$_2$ and \eqref{MHD-BOUS}$_4$, one derives that
\[
\p_t\Gamma+(u^r\p_r+u^z\p_z)\Gamma=\left(\Delta-\frac{2}{r}\p_r\right)\Gamma.
\]
\[
\p_tH+(u^r\p_r+u^z\p_z)H=0.
\]

At the beginning, the following Lemma states fundamental estimates of the system \eqref{MHD-BOUS}:
\begin{lemma}[Fundamental Energy Estimates]\label{PROP2.1}
Let $(u,h,\rho)$ be a smooth solution of \eqref{MHD-BOUS}, the we have

(i) for $p\in[1,\infty]$ and $t\in\mathbb{R}_+$,
\be\label{4.2}
\begin{split}
\|\Gamma(t,\cdot)\|_{L^p}&\leq\|\Gamma_0\|_{L^p};\\
\|H(t,\cdot)\|_{L^p}&\leq\|H_0\|_{L^p};\\
\|\rho(t,\cdot)\|_{L^p}&\leq\|\rho_0\|_{L^p}.
\end{split}
\ee

(ii) for $u_0,h_0,\rho_0\in L^2$ and $t\in\mathbb{R}_+$,
\be\label{4.3}
\|(u,h)(t,\cdot)\|^2_{L^2}+\int_0^t\|\nabla u(s,\cdot)\|_{L^2}^2ds\leq C_0(1+t)^2,
\ee
where $C_0$ depends only on $\|(u_0,h_0,\rho_0)\|_{L^2}$.
\end{lemma}
\pf
The estimate in \eqref{4.2} is classical for the heat equation when $p<\infty$ and follows from the maximum principle when $p=\infty$. Meanwhile, \eqref{4.3} follows from the standard $L^2$ estimate of the system \eqref{MB}, together with the result in \eqref{4.2}. See also \cite[Proposition 2.1]{Pan2019}. We omit the details here.

\qed
\subsection{}\textbf{$\boldsymbol{L^\infty_TL^2\cap L^2_TH^1}$ estimate of a reformulated system}\\[2mm]
First we see the vorticity $w$ of the axially symmetric velocity $u$ is defined by
\[
w=\nabla\times u=w^r(t,r,z)e_r+w^\th(t,r,z) e_\th+w^z(t,r,z)e_z,
\]
where
\[
w^r=-\p_zu^\th,\quad w^\th=\p_zu^r-\p_ru^z,\quad w^z=\p_ru^\th+\frac{u^\th}{r}.
\]
By the first three equations of \eqref{MHD-BOUS}, $(w^r,w^\th,w^z)$ satisfies
\be\label{EQVOR}
\lt\{
\begin{aligned}
&\p_t w^r+(u^r\p_r+u^z\p_z)w^r=\left(\Dl-\f{1}{r^2}\right)w^r+(w^r\p_r+w^z\p_z)u^r,\\[4mm]
&\p_t w^\th+(u^r\p_r+u^z\p_z)w^\th=\left(\Dl-\f{1}{r^2}\right)w^\th+\f{u^r}{r}w^\th+\f{1}{r}\p_z(u^\th)^2-\f{1}{r}\p_z(h^\th)^2-\p_r\rho,\\[4mm]
&\p_t w^z+(u^r\p_r+u^z\p_z)w^z=\Dl w^z+(w^r\p_r+w^z\p_z)u^z.
\end{aligned}
\rt.
\ee
 Applying $\mathcal{L}=\left(\Delta+\frac{2}{r} \partial_{r}\right)^{-1} \frac{\partial_{r}}{r}$ to the equation of $\rho$, one derives
\be\label{Lrho}
\p_t\mathcal{L}\rho+u\cdot\nabla\mathcal{L}\rho=-[\mathcal{L},u\cdot\nabla]\rho.
\ee
Meanwhile, $\eqref{EQVOR}_{2}$ indicates $\O:=\f{w^\th}{r}$ satisfies
\be\label{Omega}
\p_t\O+u\cdot\nabla\O=\left(\Delta+\f{2}{r}\p_r\right)\O-\p_zH^2-\f{\p_r\rho}{r}-\f{2u^\th w^r}{r^2}.
\ee
Now we denote $L:=\O-\mathcal{L}\rho$. Subtracting \eqref{Lrho} from \eqref{Omega} and noting the axially symmetric condition, we have
\be\label{EQL}
\p_t L+(u^r\p_r+u^z\p_z)L=\left(\Delta+\f{2}{r}\p_r\right)L-\p_zH^2+[\mathcal{L},u\cdot\nabla]\rho-2\f{u^\th w^r}{r^2}.
\ee
On the other hand, by denoting $J:=\f{w^r}{r}$, we can get the following equation from $\eqref{EQVOR}_{1}$:
\be\label{EQJ}
\p_t J+(u^r\p_r+u^z\p_z)J=\left(\Delta+\f{2}{r}\p_r\right)J+(w^r\p_r+w^z\p_z)\f{u^r}{r}.
\ee
Therefore, we have the following reformulated system by combining \eqref{EQL} and \eqref{EQJ}:
\be\label{REFORM}
\lt\{
\begin{aligned}
&\p_t L+(u^r\p_r+u^z\p_z)L=\left(\Delta+\f{2}{r}\p_r\right)L-\p_zH^2+[\mathcal{L},u\cdot\nabla]\rho-2\f{u^\th}{r}J,\\
& \p_t J+(u^r\p_r+u^z\p_z)J=\left(\Delta+\f{2}{r}\p_r\right)J+(w^r\p_r+w^z\p_z)\f{u^r}{r}.
\end{aligned}
\rt.
\ee

 Now we are ready for an a prior $L^\infty_TL^2\cap L^2_TH^1$ estimate for the above reformulated system. We have the following Lemma.
%For a clue, we present here the following observation on the vortex stretching term on the cylindrical coordinates before the proof.
%\begin{equation}\begin{aligned}
%\int_{\mathbb{R}^{3}}(w \cdot \nabla u) \cdot w d x=& \int_{\mathbb{R}^{3}} w^{r} \partial_{r} u^{r} w^{r}+w^{r} \partial_{r} u^{\theta} w^{\theta}+w^{r} \partial_{r} u^{3} w^{3}+\frac{1}{r} w^{\theta} u^{r} w^{\theta}-\frac{1}{r} w^{\theta} u^{\theta} w^{r} \\
%&+w^{3} \partial_{3} u^{r} w^{r}+w^{3} \partial_{3} u^{\theta} w^{\theta}+w^{3} \partial_{3} u^{3} w^{3} d x.\\
%\end{aligned}\end{equation}
%Since each term on the right hand side above has at least one $w^\th$, or $\bar{\nabla}b$ (which can be generated from $w^\th e_\th$ by acting a singular integral operator on it.), we only need the Prodi-Serrin-type condition for the $w^\th$ component.
\begin{lemma}\label{3.4VOL}
Under the same conditions as Theorem \ref{th0}, the following a priori estimate of $(L,J)$ holds:
\be\label{3.4VO}
\sup_{0\leq t\leq T_*}\|(L,J)(t,\cdot)\|_{L^2}^2+\int_0^{T_*}\|\nabla (L,J)(t,\cdot)\|_{L^2}^2dt<\infty.
\ee
\end{lemma}
\pf Performing the $L^2$ inner product of $\eqref{REFORM}_1$, using integration by parts and divergence-free condition, one finds
\begin{equation}\label{3.55555}
\begin{aligned}
& \frac{1}{2} \frac{d}{d t}\|L(t,\cdot)\|_{L^{2}}^{2}+\|\nabla L(t,\cdot)\|_{L^{2}}^{2} \\
\leq & \int_{\mathbb{R}^{3}} \mathcal{L}(u \cdot \nabla \rho) L d x-\int_{\mathbb{R}^{3}} u \cdot \nabla(\mathcal{L} \rho) L d x-\int_{\mathbb{R}^{3}} \partial_{z} H^{2} L d x-2\int_{\mathbb{R}^{3}} \f{u^\th}{r}JLd x \\
= & \int_{\mathbb{R}^{3}} \mathcal{L}(u \cdot \nabla \rho) L d x+\int_{\mathbb{R}^{3}}(\mathcal{L} \rho) u \cdot \nabla L d x+\int_{\mathbb{R}^{3}} H^{2} \partial_{z} L d x-2\int_{\mathbb{R}^{3}} \f{u^\th}{r}JLd x \\
:=&I_{1}+I_{2}+I_{3}+I_4.
\end{aligned}\end{equation}

Using the method in the proof of Proposition 2.2 of \cite{Pan2019}, the first 3 terms above can be estimated by
\be\label{3.66666}
\begin{split}
\sum_{j=1}^3I_j\leq& C_{h_0,\rho_0}\left(1+\|\nabla u(t,\cdot)\|_{L^{2}}^{2}\right)+C(1+t)^{2} \\
&+ C_{\rho_0}\left(1+\|L(t,\cdot)\|_{L^{2}}^{2}\right)+\f{1}{4}\|\nabla L(t,\cdot)\|_{L^2}^2.
\end{split}
\ee
Meanwhile, using the Cauchy-Schwartz inequality, $I_4$ can be estimated by
\be\label{3.88888}
\begin{split}
I_4&\leq\f{1}{2}\int_{\mathbb{R}^3}\f{|u^\th|}{r}|L|^2dx+\f{1}{2}\int_{\mathbb{R}^3}\f{|u^\th|}{r}|J|^2dx\\
&:=I_{41}+I_{42}.
\end{split}
\ee
For any $p\in[\f{3}{1+s},\infty]$, we estimate $I_{41}$ and $I_{42}$ in the following 2 cases:\\[2mm]

\leftline{\textbf{Case I: $\boldsymbol{ 0\leq s\leq1}$}\\[1mm]}
We use H\"older inequality to derive
\be\label{EEE1}
I_{41}=\int_{\mathbb{R}^3}\f{u^\th}{r^s}\f{|L|^2}{r^{1-s}}dx\leq\left\|\f{u^\th}{r^s}(t,\cdot)\right\|_{L^p}\left(\int_{\mathbb{R}^3}\left|\f{L^{2p'}}{r^{(1-s)p'}}\right|dx\right)^{1/p'},
\ee
where $p'=\f{p}{p-1}$ is the conjugate number of $p$. By choosing $\theta=(1-s)p^\prime,\ q_\ast=2p^\prime$ in \eqref{CFZEST} of Lemma \ref{CFZ}, it follows that
\be\label{EEE2}
\left(\int_{\mathbb{R}^3}\left|\f{L^{2p'}}{r^{(1-s)p'}}\right|dx\right)^{1/p'}\leq C_{s,p}\|L(t,\cdot)\|_{L^2}^{1+s-\f{3}{p}}\|\nabla L(t,\cdot)\|_{L^2}^{1-s+\f{3}{p}}.
\ee
Substituting \eqref{EEE2} in \eqref{EEE1} and using Young inequality, it follows that
\be\label{I411}
I_{41}\leq \lt\{
\begin{aligned}
& C_{s,p}\left\|\f{u^\th}{r^s}(t,\cdot)\right\|_{L^p}^{\f{2p}{(1+s)p-3}}\|L(t,\cdot)\|_{L^2}^2+\f{1}{4}\|\nabla L(t,\cdot)\|_{L^2}^2,\quad\text{for } p>\f{3}{1+s};\\[3mm]
&C_{s}\left\|\f{u^\th}{r^s}(t,\cdot)\right\|_{L^p}\|\nabla L(t,\cdot)\|_{L^2}^2,\quad\text{for }p=\f{3}{1+s} .
\end{aligned}
\rt.
\ee
Similarly, $I_{42}$ satisfies
\be\label{I421}
I_{42}\leq\lt\{
\begin{aligned}
& C_{s,p}\left\|\f{u^\th}{r^s}(t,\cdot)\right\|_{L^p}^{\f{2p}{(1+s)p-3}}\|J(t,\cdot)\|_{L^2}^2+\f{1}{4}\|\nabla J(t,\cdot)\|_{L^2}^2,\quad\text{for } p>\f{3}{1+s};\\[3mm]
& C_{s}\left\|\f{u^\th}{r^s}(t,\cdot)\right\|_{L^p}\|\nabla J(t,\cdot)\|_{L^2}^2,\quad\text{for }p=\f{3}{1+s} .
\end{aligned}
\rt.
\ee\\[1mm]

\begin{remark}
Actually the  above estimate in Case I is also feasible for $-1<s<0$. However we do not pursue it because the following $L^\infty_TL^2\cap L^2_TH^1$ estimate of $J$ fails in this situation.
\end{remark}

\qed

\leftline{\textbf{Case II: $\boldsymbol{s>1}$}\\[1mm]}
Using H\"older inequality and \eqref{4.2} in Lemma \ref{PROP2.1}, one finds
\be\label{EEE3}
I_{41}=\int_{\mathbb{R}^3}|ru^\th|^{\f{s-1}{s+1}}\left|\f{u^\th}{r^s}\right|^{\f{2}{1+s}}|L|^2dx\leq\|\Gamma_0\|_{L^\infty}^{\f{s-1}{s+1}}
\left\|\f{u^\th}{r^s}(t,\cdot)\right\|_{L^p}^{\f{2}{1+s}}\|L(t,\cdot)\|_{L^{\f{2p(1+s)}{p(1+s)-2}}}^2.
\ee
Noting that $\f{2p(1+s)}{p(1+s)-2}\in[2,6]$ when $p\geq\f{3}{1+s}$ and applying Lemma \ref{LEMGN}, one has
\be\label{EEE4}
\|L(t,\cdot)\|_{L^{\f{2p(1+s)}{p(1+s)-2}}}\leq C_{s,p}\|L(t,\cdot)\|_{L^2}^{1-\f{3}{p(1+s)}}\|\nabla L(t,\cdot)\|_{L^2}^{\f{3}{p(1+s)}}.
\ee
Thus by inserting \eqref{EEE4} into \eqref{EEE3} and using H\"{o}lder's inequality, the estimate \eqref{I411} is still valid for $s>1$ with the constant $C$ depending on  $s$, $p$ and $\|\G_0\|_{L^\i}$. The proof of \eqref{I421} when $s>1$ is similar. This finishes the estimate of $I_4$ in \eqref{3.55555}. Plugging \eqref{3.66666}, \eqref{I411} and \eqref{I421} into \eqref{3.55555}, we have the following estimate of $L$ when $p>\f{3}{1+s}$:
\be\label{ESTL1}
\begin{split}
\f{d}{dt}\|L(t,\cdot)\|_{L^2}^2+\|\nabla L(t,\cdot)\|_{L^2}^2\leq &C_{ h_0,\rho_0}\left(\|\nabla u(t,\cdot)\|_{L^2}^2+(1+t)^2+\|L(t,\cdot)\|_{L^2}^2\right)\\
&+C_{s,p,\Gamma_0}\left\|\f{u^\th}{r^s}(t,\cdot)\right\|_{L^p}^{\f{2p}{(1+s)p-3}}\left(\|L(t,\cdot)\|_{L^2}^2+\|J(t,\cdot)\|_{L^2}^2\right)\\
&+\f{1}{4}\|\nabla L(t,\cdot)\|_{L^2}^2+\f{1}{4}\|\nabla J(t,\cdot)\|_{L^2}^2,
\end{split}
\ee
and the following estimate when $p=\f{3}{1+s}$:
\be\label{ESTL2}
\begin{split}
\f{d}{dt}\|L(t,\cdot)\|_{L^2}^2+\|\nabla L(t,\cdot)\|_{L^2}^2\leq &C_{h_0,\rho_0}\left(\|\nabla u(t,\cdot)\|_{L^2}^2+(1+t)^2+\|L(t,\cdot)\|_{L^2}^2\right)\\
&+C_{s,\Gamma_0}\left\|\f{u^\th}{r^s}(t,\cdot)\right\|_{L^p}\left(\|\nabla L(t,\cdot)\|_{L^2}^2+\|\nabla J(t,\cdot)\|_{L^2}^2\right).
\end{split}
\ee
Next we work on the equation of $J$ in \eqref{REFORM}. Taking $L^2$ inner product of $\eqref{REFORM}_2$, we arrive
\[
\begin{split}
\f{1}{2}\f{d}{dt}\|J(t,\cdot)\|_{L^2}^2+\|\nabla J(t,\cdot)\|_{L^2}^2&=\int_{\mathbb{R}^3}\left(\nabla\times(u^\th e_\th)\right)\cdot\left(J\nabla\f{u^r}{r}\right)dx\\
&=\int_{\mathbb{R}^3}u^\th e_\th\cdot\left(\nabla J\times\nabla\f{u^r}{r}\right)dx\\
&=\int_{\mathbb{R}^3}u^\th\left(\p_r\f{u^r}{r}\p_zJ-\p_z\f{u^r}{r}\p_rJ\right)dx\\
&\leq \f{1}{2}\int_{\mathbb{R}^3}|u^\th|^2\left|\nabla\f{u^r}{r}\right|^2dx+\f{1}{2}\|\nabla J(t,\cdot)\|_{L^2}^2,
\end{split}
\]
which follows that
\be\label{3.8111}
\begin{split}
\f{d}{dt}\|J(t,\cdot)\|_{L^2}^2+\|\nabla J(t,\cdot)\|_{L^2}^2&\lesssim\int_{\mathbb{R}^3}|u^\th|^2\left|\nabla\f{u^r}{r}\right|^2dx.\\
%&\les\|(u^\th)^2\|_{L^{p/2}}\left\|\nabla\f{u^r}{r}\right\|_{L^2}^\al\left\|\nabla\f{u^r}{r}\right\|_{L^6}^{2-\al}
\end{split}
\ee

Different from \eqref{3.88888}, here we should be very careful to avoid the appearance of second-order gradients of $\f{u^r}{r}$. Even though the following estimate
\[
\left\|\nabla^2\f{u^r}{r}\right\|_{L^2}\leq C\|\p_z\O\|_{L^2}
\]
holds (see \cite{Lei2015-2}, equation (A.6)), we still have no idea to bound $\|\p_z\O\|_{L^2}$ due to the appearance of $\nabla\mathcal{L}\rho$ at the moment. Therefore $\nabla^2\f{u^r}{r}$ term cannot be eliminated by $\|\nabla L(t,\cdot)\|_{L^2}^2$ on the left hand side of \eqref{3.55555}. This is, in the authors' opinion, a key difference from the case of Navier-Stokes and MHD systems in which $\rho\equiv 0$.

 Nevertheless, noting that $\Gamma=ru^\th$ is uniformly bounded according to Lemma \ref{PROP2.1}, for a fixed $s\in  [0,\infty)$, \eqref{3.8111} indicates that
\[
\begin{split}
\f{d}{dt}\|J(t,\cdot)\|_{L^2}^2+\|\nabla J(t,\cdot)\|_{L^2}^2&\leq C\|\Gamma_0\|_{L^\infty}^{\f{2s}{1+s}}\int_{\mathbb{R}^3}\left|\f{u^\th}{r^s}\right|^{\f{2}{s+1}}\left|\nabla\f{u^r}{r}\right|^2dx.
\end{split}
\]
For any $p\geq\f{3}{1+s}$, using H\"older inequality, it follows that
\[
\f{d}{dt}\|J(t,\cdot)\|_{L^2}^2+\|\nabla J(t,\cdot)\|_{L^2}^2\leq C\|\Gamma_0\|_{L^\infty}^{\f{2s}{1+s}}\left\|\f{u^\th}{r^s}(t,\cdot)\right\|_{L^p}^{\f{2}{1+s}}\left\|\nabla\f{u^r}{r}\right\|_{L^2}^{2-\f{6}{(1+s)p}}\left\|\nabla\f{u^r}{r}\right\|_{L^6}^{\f{6}{(1+s)p}}.
\]
By Lemma \ref{LEMET12} and the definition of $L$, one notes that
\[
\begin{split}
&\f{d}{dt}\|J(t,\cdot)\|_{L^2}^2+\|\nabla J(t,\cdot)\|_{L^2}^2\\
\leq& C_{s,p}\|\Gamma_0\|_{L^\infty}^{\f{2s}{1+s}}\left\|\f{u^\th}{r^s}(t,\cdot)\right\|_{L^p}^{\f{2}{1+s}}\left\|\O(t,\cdot)\right\|_{L^2}^{2-\f{6}{(1+s)p}}\left\|\O(t,\cdot)\right\|_{L^6}^{\f{6}{(1+s)p}}\\
\leq & C_{s,p}\|\Gamma_0\|_{L^\infty}^{\f{2s}{1+s}}\left\|\f{u^\th}{r^s}(t,\cdot)\right\|_{L^p}^{\f{2}{1+s}}\left(\left\|L(t,\cdot)\right\|_{L^2}+\|\mathcal{L}\rho(t,\cdot)\|_{L^2}\right)^{2-\f{6}{(1+s)p}}\left(\left\|L(t,\cdot)\right\|_{L^6}+\|\mathcal{L}\rho(t,\cdot)\|_{L^6}\right)^{\f{6}{(1+s)p}}.
\end{split}
\]
Applying the boundedness of the operator $\mathcal{L}$ in Lemma \ref{LEMET11}, together with the time-uniform estimate of $\rho$ in Lemma \ref{PROP2.1}, it follows that when $p>\f{3}{1+s}$:
\be\label{ESTJ1}
\begin{split}
&\f{d}{dt}\|J(t,\cdot)\|_{L^2}^2+\|\nabla J(t,\cdot)\|_{L^2}^2\\
\leq& C_{s,p}\|\Gamma_0\|_{L^\infty}^{\f{2s}{1+s}}\left\|\f{u^\th}{r^s}(t,\cdot)\right\|_{L^p}^{\f{2}{1+s}}\left(\left\|L(t,\cdot)\right\|_{L^2}+\|\rho_0\|_{L^2}\right)^{2-\f{6}{(1+s)p}}\left(\left\|\nabla L(t,\cdot)\right\|_{L^2}+\|\rho_0\|_{L^6}\right)^{\f{6}{(1+s)p}}\\
\leq & C_{ s,p,\Gamma_0}\left\|\f{u^\th}{r^s}(t,\cdot)\right\|_{L^p}^{\f{2p}{(1+s)p-3}}\left(\left\|L(t,\cdot)\right\|_{L^2}+\|\rho_0\|_{L^2}\right)^{2}+\f{1}{4}\left(\left\|\nabla L(t,\cdot)\right\|_{L^2}+\|\rho_0\|_{L^6}\right)^{2}\\
\leq &C_{ s,p,\Gamma_0}\left\|\f{u^\th}{r^s}(t,\cdot)\right\|_{L^p}^{\f{2p}{(1+s)p-3}}\left\|L(t,\cdot)\right\|_{L^2}^2+C_{ s,p,\rho_0,\Gamma_0}\left(\left\|\f{u^\th}{r^s}(t,\cdot)\right\|_{L^p}^{\f{2p}{(1+s)p-3}}+1\right)\\
 &+\f{1}{4}\left\|\nabla L(t,\cdot)\right\|_{L^2}^2.
\end{split}
\ee
 Similarly when $p=\f{3}{1+s}$, one derives
\be\label{ESTJ2}
\begin{split}
\f{d}{dt}\|J(t,\cdot)\|_{L^2}^2+\|\nabla J(t,\cdot)\|_{L^2}^2\leq& C_{ s,\Gamma_0}\left\|\f{u^\th}{r^s}(t,\cdot)\right\|_{L^{\f{3}{1+s}}}^{\f{2}{1+s}}\left\|\nabla L(t,\cdot)\right\|_{L^2}^2+C_{ s,\rho_0,\Gamma_0}\left\|\f{u^\th}{r^s}(t,\cdot)\right\|_{L^{\f{3}{1+s}}}^{\f{2}{1+s}}.
\end{split}
\ee
Therefore, when $p>\f{3}{s+1}$, \eqref{ESTL1} and \eqref{ESTJ1} imply that
\[
\begin{split}
&\f{d}{dt}\left(\|L(t,\cdot)\|_{L^2}^2+\|J(t,\cdot)\|_{L^2}^2\right)+\left(\|\nabla L(t,\cdot)\|_{L^2}^2+\|\nabla J(t,\cdot)\|_{L^2}^2\right)\\
\leq&C_{ s,p,\Gamma_0}\left\|\f{u^\th}{r^s}(t,\cdot)\right\|_{L^p}^{\f{2p}{(1+s)p-3}}\left(\|L(t,\cdot)\|_{L^2}^2+\|J(t,\cdot)\|_{L^2}^2\right)+C_{ s,p,\rho_0,\Gamma_0}\left(\left\|\f{u^\th}{r^s}(t,\cdot)\right\|_{L^p}^{\f{2p}{(1+s)p-3}}+1\right)\\
&+C_{ h_0,\rho_0}\left(\|\nabla u(t,\cdot)\|_{L^2}^2+(1+t)^2+\|L(t,\cdot)\|_{L^2}^2\right).
\end{split}
\]
Thus the condition \eqref{C1.305} and Gronwall inequality indicates \eqref{3.4VO}. Finally when $p=\f{3}{1+s}$, \eqref{ESTL2} and \eqref{ESTJ2} lead to
\be\label{elj1st01}
\begin{split}
&\f{d}{dt}\left(\|L(t,\cdot)\|_{L^2}^2+\|J(t,\cdot)\|_{L^2}^2\right)+\left(\|\nabla L(t,\cdot)\|_{L^2}^2+\|\nabla J(t,\cdot)\|_{L^2}^2\right)\\
\leq&C_{s,\Gamma_0}\left\|\f{u^\th}{r^s}(t,\cdot)\right\|_{L^{\f{3}{1+s}}}\left(\|\nabla L(t,\cdot)\|_{L^2}^2+\|\nabla J(t,\cdot)\|_{L^2}^2\right)+C_{s,\Gamma_0}\left\|\f{u^\th}{r^s}(t,\cdot)\right\|_{L^{\f{3}{1+s}}}^{\f{2}{1+s}}\left\|\nabla L(t,\cdot)\right\|_{L^2}^2\\
&+C_{s,\rho_0,\Gamma_0}\left\|\f{u^\th}{r^s}(t,\cdot)\right\|_{L^{\f{3}{1+s}}}^{\f{2}{1+s}}+C_{ h_0,\rho_0}\left(\|\nabla u(t,\cdot)\|_{L^2}^2+(1+t)^2+\|L(t,\cdot)\|_{L^2}^2\right).
\end{split}
\ee
Choosing $\ve_0=\left(4C_{s,\Gamma_0}\right)^{-\max\left\{1,\f{1+s}{2}\right\}}$, thus the first and second terms on the right hand of \eqref{elj1st01} can be absorbed by the left hand providing
\[
\left\|\f{u^\th}{r^s}\right\|_{L^\infty(0,T_*;L^{\f{3}{1+s}})}<\ve_0.
\]
Using Gronwall inequality, \eqref{3.4VO} also holds when $p=\f{3}{1+s}$.

\qed

\begin{corollary}\label{CORL}
Under the same conditions as Theorem \ref{th0}, we have
\bes
\sup_{0\leq t\leq T_*}\|\O(t,\cdot)\|_{L^2}^2<\infty.
\ees
\end{corollary}

\pf In Lemma \ref{PROP2.1} and Lemma \ref{LEMET11}, $\mathcal{L}\rho$ satisfies:
\[
\|\mathcal{L}\rho(t,\cdot)\|_{L^2}^2\leq C\|\rho(t,\cdot)\|_{L^2}^2\leq C\|\rho_0\|_{L^2}^2<\infty,\quad\forall t\in[0,T_*].
\]
Thus the corollary is proved by noting the $L^\infty_TL^2$ boundedness of $L=\O-\mathcal{L}\rho$ in \eqref{3.4VO}.

\qed

\subsection{} \textbf{$\boldsymbol{L^\infty_TL^2\cap L^2_TH^1} $ estimate of $\boldsymbol{\nabla u}$} \\[2mm]

This part is devoted to the $L^\infty_TL^2\cap L^2_TH^1$ estimate of $\na u$, that is:
\begin{lemma}\label{3.55VOL}
Under the same conditions as Theorem \ref{th0}, the following a priori estimate of the gradient of the velocity holds:
\be\label{EVORTI}
\sup_{0\leq t\leq T_*}\|\na u(t,\cdot)\|_{L^2}^2+\int_0^{T_*}\|\nabla^2 u(t,\cdot)\|_{L^2}^2dt<\infty.
\ee
\end{lemma}

\qed

To do this, we first estimate the horizontal angular component of the vorticity.
\subsubsection{}\textbf{Estimate of $\boldsymbol{w^\th}$}\\[1mm]
\q For \eqref{EQVOR}$_2$, we perform the standard $L^2$ inner product to derive
\[
\begin{split}
&\frac{1}{2}\frac{d}{dt}\|w^\th(t,\cdot)\|_{L^2}^2+\|\nabla w^\th(t,\cdot)\|_{L^2}^2+\left\|\frac{w^\th}{r}(t,\cdot)\right\|_{L^2}^2\\
=&\int_{\mathbb{R}^3}\frac{u^r}{r}(w^\th)^2dx+\int_{\mathbb{R}^3}\p_z\frac{(u^\th)^2}{r}w^\th dx-\int_{\mathbb{R}^3}\p_r\rho w^\th dx-\int_{\mathbb{R}^3}\p_z\frac{(h^\th)^2}{r}w^\th dx\\
:=&I_1+I_2+I_3+I_4.
\end{split}
\]

Now we estimate $I_i$, $i=1,2,3,4$ separately. By H\"{o}lder inequality, Young inequality and Gagliardo-Nirenberg inequality, we have
\[
\begin{split}
I_1\leq&\|u^r(t,\cdot)\|_{L^3}\left\|\frac{w^\th}{r}(t,\cdot)\right\|_{L^2}\|w^\th(t,\cdot)\|_{L^6}\\
\leq&C\|u^r(t,\cdot)\|_{L^3}\|\Omega(t,\cdot)\|_{L^2}\|\nabla w^\th(t,\cdot)\|_{L^2}\\
\leq&C\|u^r(t,\cdot)\|_{L^3}^2\|\O(t,\cdot)\|_{L^2}^2+\frac{1}{8}\|\nabla w^\th(t,\cdot)\|_{L^2}^2\\
\leq&C\|u^r(t,\cdot)\|_{L^2}\|\nabla u^r(t,\cdot)\|_{L^2}\|\O(t,\cdot)\|_{L^2}^2+\frac{1}{8}\|\nabla w^\th(t,\cdot)\|_{L^2}^2,
\end{split}
\]
and
\[
\begin{split}
I_2&=2\int_{\mathbb{R}^3}\f{\p_zu^\th}{r}u^\th w^\th dx\\
&\leq 2\|J(t,\cdot)\|_{L^2}\|w^\th(t,\cdot)\|_{L^6}\|u^\th(t,\cdot)\|_{L^3}\\
&\leq C\|J(t,\cdot)\|_{L^2}\|\nabla w^\th(t,\cdot)\|_{L^2}\|u^\th(t,\cdot)\|_{L^3}\\
&\leq C\|u^\th(t,\cdot)\|_{L^3}^2\|J(t,\cdot)\|_{L^2}^2+\frac{1}{8}\|\nabla w^\th(t,\cdot)\|_{L^2}^2\\
&\leq C\|u^\th(t,\cdot)\|_{L^2}\|\nabla u^\th(t,\cdot)\|_{L^2}\|J(t,\cdot)\|_{L^2}^2+\frac{1}{8}\|\nabla w^\th(t,\cdot)\|_{L^2}^2.
\end{split}
\]
Meanwhile, one derives the following for $I_3$:
\[
\begin{split}
I_3=&-2\pi\int_{-\infty}^\infty\int_0^\infty\p_r\rho w^\th rdrdz\\
=&2\pi\int_{-\infty}^\infty\int_0^\infty\rho\p_r(w^\th r)drdz\\
=&2\pi\int_{-\infty}^\infty\int_0^\infty\rho\p_rw^\th rdrdz+2\pi\int_{-\infty}^\infty\int_0^\infty\rho\frac{w^\th}{r}rdrdz\\
\leq&\|\rho(t,\cdot)\|_{L^2}\|\nabla w^\th(t,\cdot)\|_{L^2}+\|\rho(t,\cdot)\|_{L^2}\left\|\frac{w^\th}{r}(t,\cdot)\right\|_{L^2}\\
\leq&C\|\rho(t,\cdot)\|_{L^2}^2+\frac{1}{8}\left(\|\nabla w^\th(t,\cdot)\|_{L^2}^2+\left\|\frac{w^\th}{r}(t,\cdot)\right\|_{L^2}^2\right),
\end{split}
\]
also similarly for $I_4$:
\[
\begin{split}
I_4=&\int_{\mathbb{R}^3}\frac{(h^\th)^2}{r}\p_zw^\th dx\\
\leq&\|H(t,\cdot)\|_{L^\infty}\|h^\th(t,\cdot)\|_{L^2}\|\nabla w^\th(t,\cdot)\|_{L^2}\\
\leq&C\|H(t,\cdot)\|_{L^\infty}^2\|h^\th(t,\cdot)\|_{L^2}^2+\frac{1}{8}\|\nabla w^\th(t,\cdot)\|_{L^2}^2.
\end{split}
\]
The above estimates for $I_i$, $i=1,2,3,4$ along with Lemma \ref{PROP2.1} indicate that
\[
\begin{split}
&\frac{d}{dt}\|w^\th(t,\cdot)\|_{L^2}^2+\|\nabla w^\th(t,\cdot)\|_{L^2}^2+\left\|\frac{w^\th}{r}(t,\cdot)\right\|_{L^2}^2\\
\leq&C\left(\|u^r(t,\cdot)\|_{L^2}\|\nabla u^r(t,\cdot)\|_{L^2}\|\O(t,\cdot)\|_{L^2}^2+\|u^\th(t,\cdot)\|_{L^2}\|\nabla u^\th(t,\cdot)\|_{L^2}\|J(t,\cdot)\|_{L^2}^2\right.\\
&\left.+\|\rho(t,\cdot)\|_{L^2}^2+\|H(t,\cdot)\|_{L^\infty}^2\|h^\th(t,\cdot)\|_{L^2}^2\right).\\
\end{split}
\]

Integrating with $t$ on $[0,T_*]$, the following final inequality follows from the $L^\infty_TL^2$ estimates of $u$, $h$ and $\rho$, together with $L^\infty_TL^\infty$ estimate of $H$ in Lemma \ref{PROP2.1}, and $L^\infty_TL^2$ estimate of $(\O, J)$ in  Lemma \ref{3.4VOL} and Corollary \ref{CORL}. That is:
\be\label{6.8}
\begin{split}
&\sup_{0\leq t\leq T_*}\|w^\th(t,\cdot)\|_{L^2}^2+\int_0^{T_*}\|\nabla w^\th(t,\cdot)\|_{L^2}^2dt+\int_0^{T_*}\left\|\frac{w^\th}{r}(t,\cdot)\right\|_{L^2}^2dt\\
\les&\sup_{0\leq t\leq T_*}\|u(t,\cdot)\|_{L^2}\sup_{0\leq t\leq T_*}\left(\|\O(t,\cdot)\|_{L^2}^2+\|J(t,\cdot)\|_{L^2}^2\right)\int_0^{T_*}\|\nabla u(t,\cdot)\|_{L^2}^2dt+T_*\|\rho_0\|_{L^2}^2\\
&+\|H_0\|_{L^\infty}^2T_*\sup_{0\leq t\leq T_*}\|h^\th(t,\cdot)\|_{L^2}^2\\
<&\infty.
\end{split}
\ee
\subsubsection{}\textbf{Estimate of $\boldsymbol{w^r}$ and $\boldsymbol{w^z}$}\\[1mm]
\q We multiply \eqref{EQVOR}$_1$ by $w^r$ and integrate over $\mathbb{R}^3$ to derive
\be\label{6.9}
\bali
&\frac{1}{2}\frac{d}{dt}\|w^r(t,\cdot)\|_{L^2}^2+\|\nabla w^r(t,\cdot)\|_{L^2}^2+\left\|\frac{w^r}{r}(t,\cdot)\right\|_{L^2}^2\\
=&\int_{\mathbb{R}^3}w^r(w^r\p_r+w^z\p_z)u^rdx\\
=&-\int_{\mathbb{R}^3}u^r(w^r\p_r+w^z\p_z)w^rdx\\
\leq\,&\|u^r(t,\cdot)\|_{L^\infty}\left(\|w^r(t,\cdot)\|_{L^2}\|\p_rw^r(t,\cdot)\|_{L^2}+\|w^z(t,\cdot)\|_{L^2}\|\p_zw^r(t,\cdot)\|_{L^2}\right)\\
\leq\,&\frac{1}{4}\|\nabla w^r(t,\cdot)\|_{L^2}^2+C\|u^r(t,\cdot)\|_{L^\infty}^2\left(\|w^r(t,\cdot)\|_{L^2}^2+\|w^z(t,\cdot)\|_{L^2}^2\right).
\eali
\ee

Here the last three lines follow from the integration by parts, H\"{o}lder inequality and Young inequality. Meanwhile, by a similar performance on \eqref{EQVOR}$_3$, one has
\be\label{6.10}
\begin{split}
&\frac{1}{2}\frac{d}{dt}\|w^z(t,\cdot)\|^2_{L^2}+\|\nabla w^z(t,\cdot)\|_{L^2}^2\\
=&\int_{\mathbb{R}^3}w^z(w^r\p_r+w^z\p_z)u^zdx\\
=&-\int_{\mathbb{R}^3}u^z(w^r\p_r+w^z\p_z)w^zdx\\
\leq&\|u^z(t,\cdot)\|_{L^\infty}\left(\|w^r(t,\cdot)\|_{L^2}\|\p_rw^z(t,\cdot)\|_{L^2}+\|w^z(t,\cdot)\|_{L^2}\|\p_zw^z(t,\cdot)\|_{L^2}\right)\\
\leq&\frac{1}{4}\|\nabla w^z(t,\cdot)\|_{L^2}^2+C\|u^z(t,\cdot)\|_{L^\infty}^2\left(\|w^r(t,\cdot)\|_{L^2}^2+\|w^z(t,\cdot)\|_{L^2}^2\right).
\end{split}
\ee

Summing up \eqref{6.9} and \eqref{6.10} and applying Gronwall inequality, one derives
\be\label{6.11}
\begin{split}
&\sup_{0\leq t\leq T_*}\left\|\left(w^r,w^z\right)(t,\cdot)\right\|_{L^2}^2+\int_0^{T_*}\left(\left\|\left(\nabla w^r(t,\cdot),\nabla w^z(t,\cdot)\right)\right\|+\left\|\frac{w^r}{r}(t,\cdot)\right\|_{L^2}\right)dt\\
\leq&\left\|\left(w^r_0,w^z_0\right)\right\|_{L^2}^2\exp\left(C\int_0^{T_*}\|b(t,\cdot)\|_{L^\infty}^2dt\right).
\end{split}
\ee

Finally, it remains to estimate the part inside the exponential function on the right-hand-side of \eqref{6.11}. Using Gagliardo-Nirenberg interpolation inequality, \eqref{Biot2} and H\"{o}lder inequality, together with estimates \eqref{4.3} and \eqref{6.8}, one has
\[
\begin{split}
&\int_0^{T_*}\|b(t,\cdot)\|^2_{L^\infty}dt\\
&\leq C\int_0^{T_*}\|\nabla b(t,\cdot)\|_{L^2}\|\nabla^2b(t,\cdot)\|_{L^2}dt\\
&\leq C\int_0^{T_*}\|\nabla u(t,\cdot)\|_{L^2}\left(\|\nabla w^\th(t,\cdot)\|_{L^2} +\left\|\frac{w^\th}{r}(t,\cdot)\right\|_{L^2}\right)dt\\
&\leq C\left(\int_0^{T_*}\|\nabla u(t,\cdot)\|^2_{L^2}ds\right)^{1/2}\left(\int_0^{T_*}\left(\|\nabla w^\th(t,\cdot)\|^2_{L^2} +\left\|\frac{w^\th}{r}(t,\cdot)\right\|^2_{L^2}\right)dt\right)^{1/2}<\infty.
\end{split}
\]
Inserting the above estimate in \eqref{6.11}, we have
\be\label{EWRZ}
\begin{split}
\sup_{0\leq t\leq T_*}\left\|\left(w^r,w^z\right)(t,\cdot)\right\|_{L^2}^2+\int_0^{T_*}\left(\left\|\nabla\left( w^r,w^z\right)(t,\cdot)\right\|_{L^2}^2+\left\|\frac{w^r}{r}(t,\cdot)\right\|^2_{L^2}\right)dt<\infty.
\end{split}
\ee
Concluding \eqref{6.8} and \eqref{EWRZ}, we have the $L_T^\infty L^2\cap L^2_{T}H^1$ estimate for the vorticity. Then using \eqref{Biot1}, \eqref{EVORTI} follows.

\subsection{}\textbf{$\boldsymbol{L^1_TL^\infty}$ estimate of $\boldsymbol{\nabla u}$}\\[2mm]

Recall the equation for the vorticity:
\bes
\lt\{
\begin{aligned}
&\p_t w-\Delta w=\nabla\times(u\cdot\nabla u)-\nabla\times(h\cdot\nabla h)+\nabla\times(\rho e_3);\\[4mm]
&w(0,x)=\nabla\times u_0(x).
\end{aligned}
\rt.
\ees
For the further convenience, we split $w$ into three parts:
\[
w:=w_0+w_1+w_2,
\]
where $w_0$ solves the linear parabolic equation with the initial value $\nabla\times u_0(x)$:
\bes
\lt\{
\begin{aligned}
&\p_t w_0-\Delta w_0=0;\\[4mm]
&w(0,x)=\nabla\times u_0(x).
\end{aligned}
\rt.\ees
Clearly, when $t>0$, $w_0$ is regular enough for our argument in this paper, so we only need to consider the rest parts. Meanwhile, $w_1$ and $w_2$, which have homogeneous initial  data, satisfy
\bes
\p_t w_1-\Delta w_1=-\nabla\times(h\cdot\nabla h)
\ees
and
\bes
\p_t w_2-\Delta w_2=\nabla\times(u\cdot\nabla u)+\nabla\times(\rho e_3),
\ees
respectively.

Now we \textbf{claim that}
\be\label{ENU}
\boldsymbol{\nabla u\in L^1\left(0,T_*;L^\infty(\mathbb{R}^3)\right)}.
\ee
To prove it, we first observe that
\[
h\cdot\nabla h=-\f{(h^\th)^2}{r}e_r=-Hh^\th e_r,
\]
since $h=h^\th(t,r,z)e_\th$. Noting that
\be\label{3.122}
H\in L^\infty\left(0,T_*;L^\infty(\mathbb{R}^3)\right)
\ee
follows from \eqref{4.2} in Lemma \ref{PROP2.1}, the following estimate of $h^\th$ holds by performing $L^4$ inner product of $h^\th${\blue :}
\[\begin{aligned}
\frac{d}{d t}\left\|h^{\theta}(t,\cdot)\right\|_{L^{4}}^{4} & \leq 4 \left|\int_{\mathbb{R}^{3}} \frac{u^{r}}{r}\left(h^{\theta}\right)^{4} d x\right| \\
& \leq 4\|H\|_{L^{\infty}} \int_{\mathbb{R}^{3}}\left|u^{r}\right||h^{\theta}|^{3} d x \\
& \leq 4\left\|H_{0}\right\|_{L^{\infty}}\left\|u^{r}(t,\cdot)\right\|_{L^{4}}\left\|h^{\theta}(t,\cdot)\right\|_{L^{4}}^{3} \\
& \leq C\left\|H_{0}\right\|_{L^{\infty}}\left\|\nabla u^{r}(t,\cdot)\right\|_{L^{2}}^{3 / 4}\left\|u^{r}(t,\cdot)\right\|_{L^{2}}^{1 / 4}\left\|h^{\theta}(t,\cdot)\right\|_{L^{4}}^{3}.
\end{aligned}\]
Integration from $0$ to $t$ on time for $t\in(0,T_*]$, one derives
\be\label{3.133}
\begin{split}
&\q\sup_{0\leq t\leq T_*}\|h^\th(t,\cdot)\|_{L^4}\\
&\leq \|h^\th_0\|_{L^4}+C\|H_0\|_{L^\infty}\sup_{0\leq t\leq T_*}\|u^r(t,\cdot)\|_{L^2}^{1/4}\int_0^{T_*}\left\|\nabla u^{r}(t,\cdot)\right\|_{L^{2}}^{3 / 4}dt\\
&\leq \|h^\th_0\|_{L^4}+C\|H_0\|_{L^\infty}\sup_{0\leq t\leq T_*}\|u^r(t,\cdot)\|_{L^2}^{1/4}\left(\int_0^{T_*}\left\|\nabla u(t,\cdot)\right\|_{L^{2}}^{2}dt\right)^{3/8}{T_*}^{5/8}\\
&<\infty.
\end{split}
\ee
Combining \eqref{3.122} and \eqref{3.133}, we find
\[
h\cdot\nabla h\in L^\infty\left(0,T_*;L^{4}(\mathbb{R}^3)\right)\subset L^{4/3}\left(0,T_*;L^{4}(\mathbb{R}^3)\right).
\]
%Denoting $\tilde{s}=\min\{s,4/3\}$ and $\tilde{r}_0=\min\{r_0,4\}$, we have
Then $\nabla w_1$ satisfies
\be\label{Ew1}
\nabla w_1\in L^{4/3}\left(0,T_*;L^{4}(\mathbb{R}^3)\right)
\ee
by applying \eqref{2.10PR}, the maximal regularity for the heat flow in Lemma \ref{MRP}. To treat $w_2$, by interpolating $L^2_TH^1$ and $L^\infty_TL^2$ as shown in \eqref{Vorti222} of Lemma \eqref{LEM2.3}, we arrive
\be\label{3.14}
\nabla u\in L^{8/3}\left(0,T_*;L^{4}(\mathbb{R}^3)\right).
\ee
Also we have the following interpolation inequality by Lemma \ref{LEMGN}:
\[
\|u(t,\cdot)\|_{L^\infty}\lesssim\|\nabla u(t,\cdot)\|_{L^{4}}^{6/7}\|u(t,\cdot)\|_{L^2}^{1/7}.
\]
%Denoting that
%\[
%\f{1}{\tilde{s}}=\f{1}{s}-\f{1}{s_0},
%\]
Then considering the fundamental energy estimate \eqref{4.3}, it follows that
\be\label{3.15}
\begin{split}
\int_0^{T_*}\|u(t,\cdot)\|_{L^\infty}^{8/3}dt&\lesssim\|u\|_{L^\infty(0,T_*;L^2)}^{8/21}\int_0^{T_*}\|\nabla u(t,\cdot)\|_{L^4}^{16/7}dt\\
&\lesssim\| u\|_{L^\infty(0,T_*;L^2)}^{8/21}\left(\int_0^{T_*}\|\nabla u(t,\cdot)\|_{L^{4}}^{8/3}dt\right)^{6/7}T_*^{1/7}<\infty.\\
\end{split}
\ee
%Here the first inequality holds due to \eqref{E000} in Corollary \ref{COR2.2}.
Then \eqref{3.14} and \eqref{3.15} assert that
\[
u\cdot\nabla u\in L^{4/3}\left(0,T_*;L^{4}(\mathbb{R}^3)\right).
\]
Meanwhile, by \eqref{4.2}, it is clear that
\[
\rho\in L^{\infty}(0,T_*;L^{4}(\mathbb{R}^3))\subset L^{4/3}(0,T_*;L^{4}(\mathbb{R}^3)).
\]
Following from \eqref{2.10PR} in Lemma \ref{MRP}, it is clear that
\be\label{Ew2}
\nabla w_2\in L^{4/3}\left(0,T_*;L^{4}(\mathbb{R}^3)\right).
\ee
Then \eqref{Ew2}, together with \eqref{Ew1}, imply that
\be\label{ESTW1}
\nabla w\in L^{4/3}\left(0,T_*;L^{4}(\mathbb{R}^3)\right).
\ee
Now the interpolation inequality in Lemma \ref{LEMGN}, together with the lower order estimate of $w$ in \eqref{3.4VO} and \eqref{Biot1}, assert that
\[
\begin{split}
\|\nabla u(t,\cdot)\|_{L^\infty}&\lesssim\|\nabla u(t,\cdot)\|^{1/7}_{L^{2}}\|\nabla^2u(t,\cdot)\|^{6/7}_{L^{4}}\\
&\lesssim\|w(t,\cdot)\|^{1/7}_{L^{2}}\|\nabla w(t,\cdot)\|^{6/7}_{L^{4}}.
\end{split}
\]
Then using \eqref{ESTW1}, we find the claim is proved since
\[
\begin{split}
\int_0^{T_*}\|\nabla u(t,\cdot)\|_{L^\infty}dt&\lesssim\|w\|^{1/7}_{L^\infty(0,T_*;L^2)}\int_0^{T_*}\|\nabla w(t,\cdot)\|_{L^{4}}^{6/7}dt\\
&\leq\|w\|^{1/7}_{L^\infty(0,T_*;L^2)}\left(\int_0^{T_*}\|\nabla w(t,\cdot)\|_{L^{4}}^{4/3}dt\right)^{14/9}T_*^{\,5/14}<\infty.
\end{split}
\]

\subsection{}\textbf{$\boldsymbol{L^1_TL^\infty}$ estimate of $\boldsymbol{\nabla\times h}$}\\[2mm]

Let $j:=\nabla\times h$. By $h=h^\th(t,r,z)e_\th$, it follows that
\[
j=j^r(t,r,z)e_r+j^z(t,r,z)e_z
\]
is an axially symmetric swirl-free vector field with
\[
j^r=-\p_zh^\th,\quad j^z=\p_rh^\th+\f{h^\th}{r}.
\]
Taking derivative of $\eqref{MHD-BOUS}_4$, noting the divergence-free condition of $u$, one obtains
\be\label{EQVORM1}
\lt\{
\begin{aligned}
&\p_t j^r+(u^r\p_r+u^z\p_z)j^r=-\left(\p_ru^r+2\p_zu^z\right)j^r+\p_zu^rj^z-2\p_zu^rH;\\[4mm]
&\p_t j^z+(u^r\p_r+u^z\p_z)j^z=\p_ru^zj^r-\left(2\p_ru^r+\p_zu^z\right)j^z+\left(4\p_ru^r+2\p_zu^z\right)H.
\end{aligned}
\rt.
\ee

 Before we perform the $L^1_TL^\infty$-estimate of $\nabla\times h$, we denote $X(t,\cdot):\,\mathbb{R}^3\to\mathbb{R}^3$ the particle trajectory mapping of the velocity $b$, which solves the initial value problem:
\bes
\f{\p X(t,\zeta)}{\p t}=b(t,X(t,\zeta)),\quad X(0,\zeta)=\zeta.
\ees
Integrating \eqref{EQVORM1} along the particle trajectory mapping, we have
%\be\label{EJR}
\[
j^r(t,X(t,\zeta))=j^r_0(\zeta)+\int_0^t\left[-\left(\p_ru^r+2\p_zu^z\right)j^r+\p_zu^rj^z-2\p_zu^rH\right](s,X(s,\zeta))ds;
\]
%\be\label{EJZ}
\[
j^z(t,X(t,\zeta))=j^z_0(\zeta)+\int_0^t\left[\p_ru^zj^r-\left(2\p_ru^r+\p_zu^z\right)j^z+\left(4\p_ru^r+2\p_zu^z\right)H\right](s,X(s,\zeta))ds.
\]
Taking the $L^\infty$ norm over $\zeta\in\mathbb{R}^3$, noting the estimate of $H$ in \eqref{4.2}, it follows that
\bes
\begin{split}
\|(j^r,j^z)(t,\cdot)\|_{L^\infty}\leq&\|(j^r_0\,,\,j^z_0)\|_{L^\infty}+C\int_0^t\|{\nabla}b(s,\cdot)\|_{L^\infty}\left(\|(j^r,j^z)(s,\cdot)\|_{L^\infty}+\|H(s,\cdot)\|_{L^\infty}\right)ds\\
\leq&\|(j^r_0\,,\,j^z_0)\|_{L^\infty}+C\|H_0\|_{L^\infty}\int_0^t\|{\nabla}b(s,\cdot)\|_{L^\infty}ds\\
&+C\int_0^t\|{\nabla}b(s,\cdot)\|_{L^\infty}\|(j^r,j^z)(s,\cdot)\|_{L^\infty}ds.
\end{split}
\ees
Applying \eqref{ENU} and Gronwall inequality, it follows that
\[
\begin{split}
&\|(j^r,j^z)(t)\|_{L^\infty}\\
\leq&\left(\|(j^r_0\,,\,j^z_0)\|_{L^\infty}+C\|H_0\|_{L^\infty}\int_0^t\|{\nabla}b(s,\cdot)\|_{L^\infty}ds\right)\exp\left(C\int_0^t\|{\nabla}b(s,\cdot)\|_{L^\infty}ds\right)\\
\leq&\left(\|(j^r_0\,,\,j^z_0)\|_{L^\infty}+C\|H_0\|_{L^\infty}\int_0^t\|\nabla u(s,\cdot)\|_{L^\infty}ds\right)\exp\left(C\int_0^t\|{\nabla}u(s,\cdot)\|_{L^\infty}ds\right)\\
\end{split}
\]
holds for any $t\in(0,T_*]$. This implies
\be\label{ENH}
\int_0^{T_*}\|\nabla\times h(t,\cdot)\|_{L^\infty}dt=\int_0^{T_*}\|(j^r,j^z)(t,\cdot)\|_{L^\infty}dt<\infty,
\ee
which finishes the proof of the desired estimate.

\subsection{} \textbf{$\boldsymbol{L^1_TL^\infty}$ estimate of $\boldsymbol{\nabla \rho}$}\ \\

Now it remains to esitmate $\nabla\rho$. Taking $\nabla$ to $\eqref{MB}_3$, we know that
\[
\p_t\nabla\rho+u\cdot\nabla\nabla\rho=-\nabla u\cdot\nabla\rho.
\]
The routine $L^\infty$ estimate follows that
\[
\|\nabla\rho(t,\cdot)\|_{L^\infty}\leq\|\nabla\rho_0\|_{L^\infty}+\int_0^t\|\nabla u(s,\cdot)\|_{L^\infty}\|\nabla\rho(s,\cdot)\|_{L^\infty}ds.
\]
By Gronwall inequality and using \eqref{ENU}, we arrive
\be\label{ERHO}
\sup_{0\leq t\leq T_*}\|\nabla\rho(t,\cdot)\|_{L^\infty}\leq\|\nabla\rho_0\|_{L^\infty}\exp\left(\int_0^{T_*}\|\nabla u(s,\cdot)\|_{L^\infty}ds\right)<\infty.
\ee

\subsection{Estimates of higher order norms \& proof of Theorem \ref{th0}}\ \\

Combining \eqref{ENU}, \eqref{ENH} and \eqref{ERHO}, we have
\be\label{e1stinf}
\int_0^{T_*}\|(\nabla\times u,\nabla\times h)(t,\cdot)\|_{L^\infty}dt+\int_0^{T_*}\|\nabla\rho(t,\cdot)\|_{L^\infty}dt<\infty.
\ee

 We now show $H^m\ (m\geq 3)$ regularity of the solution by using the above inequality. We note that the proof below is still valid for the case that the viscous coefficient $\mu=0$.

Apply $\na^m$ $(m\in\mathbb{N},\,\,m\geq3)$ to \eqref{MB}$_{1,2,3}$ to derive that
\begin{equation}\label{MB3}
\left\{
\begin{aligned}
&
\p_t\na^mu+u\cdot\nabla\na^mu+\nabla\na^mp-\mu\Delta\na^mu\\
&\hskip 1cm \qq\qq\qq=h\cdot\nabla\na^mh+\na^m(\rho e_3)-[\na^m,u\cdot\nabla]u+[\na^m,h\cdot\nabla]h,\\[4mm]
&\p_t\na^mh+u\cdot\nabla\na^m h-h\cdot\nabla\na^m u=-[\na^m,u\cdot\nabla]h+[\na^m,h\cdot\nabla]u,\\[4mm]
&\p_t\na^m\rho+u\cdot\na\na^m\rho=-[\na^m,u\cdot\nabla]\rho.\\
\end{aligned}
\right.
\end{equation}
Performing the $L^2$ energy estimate of \eqref{MB3}, noting that
\[
\int_{\mathbb{R}^3}h\cdot\nabla\nabla^mh\cdot\nabla^mudx+\int_{\mathbb{R}^3}h\cdot\nabla\nabla^mu\cdot\nabla^mhdx=0,
\]
we have
\[
\bali
\frac{1}{2}\frac{d}{dt}&\left\|\na^m (u,h,\rho)(t,\cdot)\right\|_{L^2}^2+\mu\left\|\na^{m+1}u(t,\cdot)\right\|_{L^2}^2\\
=&-\int_{\mathbb{R}^3}[\na^m,u\cdot\nabla]u\na^mudx+\int_{\mathbb{R}^3}[\na^m,h\cdot\nabla]h\na^mudx-\int_{\mathbb{R}^3}[\na^m,u\cdot\nabla]h\na^mhdx\\
&+\int_{\mathbb{R}^3}[\na^m,h\cdot\nabla]u\na^mhdx-\int_{\mathbb{R}^3}[\na^m,u\cdot\nabla]\rho\na^m\rho dx+\int_{\mathbb{R}^3}\na^m(\rho e_3)\na^mudx.
\eali
\]
By Lemma \ref{LEMET1}, the above equation implies
\be\label{ESTH}
\bali
&\frac{d}{dt}\left\|\na^m (u,h,\rho)(t,\cdot)\right\|_{L^2}^2+\mu\left\|\na^{m+1}u(t,\cdot)\right\|_{L^2}^2\\
\lesssim&\|\nabla^m(u,h,\rho)(t,\cdot)\|_{L^2}^2\left(\|\nabla(u,h,\rho)(t,\cdot)\|_{L^\infty}+1\right).
\eali
\ee
By denoting
\[
E_m(t):=\sup\limits_{0\leq s\leq t}\|\na^m(u,h,\rho)(s,\cdot)\|^2_{L^2},\q  0\leq t<T_*,\quad m\geq3,
\]
\eqref{ESTH}, together with \eqref{2.29b} in Corollary \ref{COR2.6}, \eqref{4.2}$_3$ and \eqref{4.3} in Lemma \ref{PROP2.1}, indicate that
\[%\label{4.20c}
\bali
&\f{d}{dt}\|\na^m(u,h,\rho)(t,\cdot)\|^2_{L^2}+\mu\|\na^{m+1} u(t,\cdot)\|^2_{L^2}\\
\les&\left(1+\|(\na\times u,\nabla\times h,\nabla\rho)(t,\cdot)\|_{BMO }\log(e+E_m(t))\right)(e+E_m(t))\\
\les&\left(1+\|(\na\times u,\nabla\times h,\nabla\rho)(t,\cdot)\|_{L^\infty}\log(e+E_m(t))\right)(e+E_m(t)).
\eali
\]
Integrating the above inequality over $(0,\,t)$, where $t\in[0,\,T_*)$, one has
\[
\bali
&e+\|\na^m(u,\,h,\,\rho)(t)\|_{L^2}^2\\
\les& e+\|\na^m(u_0,\,h_0,\,\rho_0)\|_{L^2}^2\\
&+\int_{0}^t\Big\{1+\|(\na\times u,\nabla\times h,\nabla\rho)(s,\cdot)\|_{L^\infty}\log\left(e+E_m(s)\right)\Big\}\left(e+E_m(s)\right)ds,
\eali
\]
which implies
\[
\bali
e+E_m(t)\les& e+\|\na^m(u_0,\,h_0,\,\rho_0)\|_{L^2}^2\\
&+\int_{0}^t\Big\{1+\|(\na\times u,\nabla\times h,\nabla\rho)(s,\cdot)\|_{L^\infty}\log\left(e+E(s)\right)\Big\}\left(e+E(s)\right)ds.
\eali
\]
Using Gronwall inequality twice, it follows that
\[
e+E_m(t)\leq  C_{0,T_*}\left(e+\|\na^m(u_0,\,h_0,\,\rho_0)\|_{L^2}^2\right)^{\exp\left( C_{0,T_*}\int^t_{0}\left(1+\|(\na\times u,\nabla\times h,\nabla\rho)(s,\cdot)\|_{L^\infty}\right)ds\right)},\quad \forall t\in[0,T_*),
\]
 where $C_{0,T_*}>0$ is a constant depends on initial data and $T_*$. Hence $(u,h,\rho)$ can be regularly extended beyond $T_*$ under the  condition \eqref{e1stinf}. This completes the proof of Theorem \ref{th0}.

%\subsection{}\textbf{Final steps in proving Theorem \ref{th0}}\\[2mm]
%We refer readers to \cite{LiPan2020-1} for detailed proof of this section. Here we only outline in the sense of completeness. Using the conclusions in the previous sections, we may first achieve a $L^1_TL^\infty$ estimate of the spacial derivatives of the velocity by using the maximal regularity of the heat flow:
%\[
%\int_0^{T_*}\|\nabla u(t,\cdot)\|_{L^\infty}dt<\infty,
%\]
%then it follows the same estimates for $\nabla\times h$ and $\nabla\rho$:
%\[
%\int_0^{T_*}\|(\nabla\times h,\nabla\rho)(t,\cdot)\|_{L^\infty}dt<\infty.
%\]
%Finally, based on those preparations, the higher-order estimates of the solution can be established by using commutator estimates and Gronwall inequality
%\[
%\sup_{0\leq t\leq T_*}\|\na^m(u,\,h,\,\rho)(t,\cdot)\|_{L^2}^2\leq \left(e+\|\na^m(u_0,\,h_0,\,\rho_0)\|_{L^2}^2\right)^{\exp\left(\int^t_{0}\left(1+\|(\na\times u,\nabla\times h,\nabla\rho)(s,\cdot)\|_{L^\infty}\right)ds\right)}
%\]
%for any $m\in\mathbb{N}$ and $m\geq 3$, which proves the solution can be smoothly extended beyond $T_*$.
\bibliographystyle{plain}

%\bibliography{/Users/dykim/Dropbox/01Research/oralpaper}

%%%%%%%%%%%%%%%%%%%%%%%%%%%%%%%%%%%%
%\begin{comment}

\def\cprime{$'$}

%\end{comment}
%%%%%%%%%%%%%%%%%%%%%%%%%%%%%%%%%%%%

\end{document}